\documentclass[reqno]{amsart}

\input xy
\xyoption{all}
\usepackage{epsfig}
\usepackage{color}
\usepackage{amsthm}
\usepackage{amssymb}
\usepackage{amsmath}
\usepackage{amscd}
\usepackage{amsopn}
\usepackage{url}
\usepackage{hyperref}\hypersetup{colorlinks}

\usepackage{color} 

\definecolor{darkred}{rgb}{1,0,0} 
\definecolor{darkgreen}{rgb}{0,0.8,0}
\definecolor{darkblue}{rgb}{0,0,1}

\hypersetup{colorlinks,
linkcolor=darkblue,
filecolor=darkgreen,
urlcolor=darkred,
citecolor=darkgreen}

\newcommand{\labell}[1] {\label{#1}}

\numberwithin{equation}{section}
\newtheorem {Theorem}{Theorem}
\numberwithin{Theorem}{section}

\newtheorem {Lemma}[Theorem]    {Lemma}

\newtheorem {Proposition}[Theorem]{Proposition}

\theoremstyle{definition}
\newtheorem{Definition}[Theorem]{Definition}
\theoremstyle{remark}
\newtheorem{Remark}[Theorem]{Remark}
\newtheorem{Example}[Theorem]{Example}

\expandafter\chardef\csname pre amssym.def at\endcsname=\the\catcode`\@
\catcode`\@=11
\def\undefine#1{\let#1\undefined}
\def\newsymbol#1#2#3#4#5{\let\next@\relax
 \ifnum#2=\@ne\let\next@\msafam@\else
 \ifnum#2=\tw@\let\next@\msbfam@\fi\fi
 \mathchardef#1="#3\next@#4#5}
\def\mathhexbox@#1#2#3{\relax
 \ifmmode\mathpalette{}{\m@th\mathchar"#1#2#3}%
 \else\leavevmode\hbox{$\m@th\mathchar"#1#2#3$}\fi}
\def\hexnumber@#1{\ifcase#1 0\or 1\or 2\or 3\or 4\or 5\or 6\or 7\or 8\or
 9\or A\or B\or C\or D\or E\or F\fi}

\font\teneufm=eufm10
\font\seveneufm=eufm7
\font\fiveeufm=eufm5
\newfam\eufmfam
\textfont\eufmfam=\teneufm
\scriptfont\eufmfam=\seveneufm
\scriptscriptfont\eufmfam=\fiveeufm

\catcode`\@=\csname pre amssym.def at\endcsname

\def    \eps    {\epsilon}

\newcommand{\CA}{{\mathcal A}}

\newcommand{\CS}{{\mathcal S}}

\newcommand{\supp}{\operatorname{supp}}

\newcommand{\tH}{\tilde{H}}

\newcommand{\bPP}{\bar{\mathcal P}}

\def    \F      {{\mathbb F}}

\def    \R      {{\mathbb R}}

\def    \Z      {{\mathbb Z}}
\def    \N      {{\mathbb N}}

\def    \CP     {{\mathbb C}{\mathbb P}}

\def    \12    {{\frac{1}{2}}}

\def    \p      {\partial}

\def    \im     {\operatorname{im}}

\def    \HF     {\operatorname{HF}}

\def    \H     {\operatorname{H}}

\def    \CF     {\operatorname{CF}}

\def    \bx     {\bar{x}}
\def    \by     {\bar{y}}
\def    \bz     {\bar{z}}

\def    \MUCZ  {\operatorname{\mu_{\scriptscriptstyle{CZ}}}}

\def    \s  {\operatorname{c}}

\def    \ssminus        {\smallsetminus}

\begin{document}


\setlength{\smallskipamount}{6pt}
\setlength{\medskipamount}{10pt}
\setlength{\bigskipamount}{16pt}





\title[Conley Conjecture]{Conley Conjecture for
  Negative Monotone Symplectic Manifolds}

\author[Viktor Ginzburg]{Viktor L. Ginzburg}
\author[Ba\c sak G\"urel]{Ba\c sak Z. G\"urel}

\dedicatory{\bigskip\normalsize{Dedicated to Edi Zehnder on the occasion of 
his seventieth birthday}}

\address{VG: Department of Mathematics, UC Santa Cruz,
Santa Cruz, CA 95064, USA}
\email{ginzburg@math.ucsc.edu}

\address{BG: Department of Mathematics, Vanderbilt University,
Nashville, TN 37240, USA} \email{basak.gurel@vanderbilt.edu}


 \date{\today} 

 \thanks{The work is partially supported by NSF and by the faculty
   research funds of the University of California, Santa Cruz.}


\begin{abstract} 

  We prove the Conley conjecture for negative monotone, closed
  symplectic manifolds, i.e., the existence of infinitely many
  periodic orbits for Hamiltonian diffeomorphisms of such manifolds.

\end{abstract}

\maketitle

\tableofcontents

\section{Introduction and main results}
\labell{sec:main-results}

\subsection{Introduction}
\label{sec:intro}
In this paper, we establish the Conley conjecture for negative
monotone, closed symplectic manifolds. More specifically, we show that
a Hamiltonian diffeomorphism of such a manifold has infinitely many
periodic orbits.

The Conley conjecture was formulated by Charles Conley in 1984 for the
torus (see \cite{Co}), and since then the conjecture has been a
subject of active research focusing on establishing the existence of
infinitely many periodic orbits for broader and broader class of
symplectic manifolds or Hamiltonian diffeomorphisms. Note that the
example of an irrational rotation of $S^2$ shows that the Conley
conjecture does not hold unconditionally -- some assumptions on the
manifold or the Hamiltonian diffeomorphism are required -- in contrast
with, say, the Arnold conjecture; see \cite{GK} for more examples of
this type.

The conjecture was proved for the so-called weakly non-degenerate
Hamiltonian diffeomorphisms in \cite{SZ} (see also \cite{CZ}) and for
all Hamiltonian diffeomorphisms of surfaces other than $S^2$ in
\cite{FrHa} (see also \cite{LeC}). In its original form, as stated in
\cite{Co} for the torus, the conjecture was established in \cite{Hi}
and the case of an arbitrary closed, symplectically aspherical
manifold was settled in \cite{Gi:conley}. The proof was extended to
rational, closed symplectic manifolds $M$ with $c_1(TM)|_{\pi_2(M)}=0$
in \cite{GG:gaps} and the rationality requirement was eliminated in
\cite{He:irr}. Thus, after \cite{SZ}, the main difficulty in
establishing the Conley conjecture for more and more general manifolds
with spherically vanishing first Chern class, overcome in this series
of works, lied in proving the conjecture for totally degenerate
Hamiltonian diffeomorphisms not covered by \cite{SZ}.  (The internal
logic in \cite{FrHa,LeC}, relying on low-dimensional dynamics methods,
is somewhat different.)

Two other variants of the Conley conjecture have also been
investigated. One is the Conley conjecture for Hamiltonian
diffeomorphisms with displaceable support; see, e.g.,
\cite{FS,Gu,HZ,schwarz,Vi:gen}. Here the manifold $M$ is required to
be symplectically aspherical, but not necessarily closed. The second
one is the Lagrangian Conley conjecture or more generally the Conley
conjecture for Hamiltonians with controlled behavior at infinity on
cotangent bundles; see \cite{He:cot,Lo:tori,Lu,Ma}.
 
As has been pointed out above, some requirement on $M$ is necessary
for the Conley conjecture to hold, but this requirement need not
necessarily be the condition that $c_1(TM)|_{\pi_2(M)}=0$. One
hypothetical replacement of this condition, conjectured by the second
author of this paper, is that the minimal Chern number $N$ of $M$ is
sufficiently large, e.g., $N>\dim M$. (The condition
$c_1(TM)|_{\pi_2(M)}=0$ corresponds to $N=\infty$.) More generally, it
might be sufficient to require that the Gromov--Witten invariants of
$M$ vanish, as suggested by Michael Chance and Dusa McDuff, or even
that the quantum product is undeformed. No results in this direction
have been proved to date. Alternatively, one may require $M$ to be
negative monotone and this is the case we consider in the present
paper. (Our result almost, but not quite, fits within the scope of the
Chance-McDuff conjecture: it is easy to see that the Gromov--Witten
invariants of a negative monotone manifold $M$ vanish when $N>\dim
M/2$; cf.\ \cite{LO}.)

A major difficulty in proving the conjecture when
$c_1(TM)|_{\pi_2(M)}\neq 0$ lies in establishing the weakly
non-degenerate, or even non-degenerate, case. For the totally
degenerate case is settled in \cite{GG:gaps} without any requirements
on the Chern class for rational manifolds. (Interestingly, once the
rationality condition is dropped, vanishing of the first Chern class
becomes again essential for the proof,~\cite{He:irr}.) Our proof
relies on keeping track of the behavior of both the index and the
action under iterations along the lines of the reasoning from
\cite{GG:gaps}, in contrast with the argument from \cite{SZ} making
use only of the index, and on the fact that for a negative monotone
manifold the index and action change in the opposite ways under
recapping. The proof also crucially relies on the subadditivity of the
action selector with respect to the pair-of-pants product. This part
of the proof is reminiscent of the argument for Hamiltonians with
displaceable support in \cite{schwarz,Vi:gen} and has no corresponding
counterpart in other proofs of the Conley conjecture; cf.\
\cite{SZ,Hi,Gi:conley,GG:gaps,He:irr}. On the technical level, the
present work draws heavily on \cite{GG:gaps} and can be thought of as
a follow up of that paper.

We also use the machinery developed in \cite{GG:gaps} to relate, in
the symplectically aspherical case, the growth of orbits to the decay
of mean indices for a certain type of periodic orbits.

\subsection{Results}
Recall that a symplectic manifold $(M,\omega)$ is called
\emph{negative monotone} if $[\omega]|_{\pi_2(M)}=\lambda
c_1(TM)|_{\pi_2(M)}$ for some $\lambda<0$. Negative monotone
symplectic manifolds exist in abundance. A standard example is the
hypersurface $z_0^m+\ldots +z_n^m=0$ in $\CP^n$, where $m>n+1$; see,
e.g., \cite[pp.\ 429--430]{MS} for the case of $n=4$.

\begin{Theorem}
\labell{thm:main}
Let $\varphi$ be a Hamiltonian diffeomorphism of a closed, negative
monotone symplectic manifold. Then $\varphi$ has infinitely many 
periodic orbits.
\end{Theorem}

\begin{Remark}
  Note that this theorem is somewhat weaker than the versions of the
  Conley conjecture established in the case where
  $c_1(TM)|_{\pi_2(M)}=0$. Namely, while here we prove only the
  existence of infinitely many periodic orbits, the results for
  $c_1(TM)|_{\pi_2(M)}=0$ assert the existence of simple periodic
  orbits with arbitrarily large period, provided that the fixed point
  set is finite.

  Furthermore, the proof of Theorem \ref{thm:main} utilizes
  Hamiltonian Floer theory. Hence, unless $M$ is required to be weakly
  monotone, the argument ultimately, although not explicitly, relies
  on the machinery of multi-valued perturbations and virtual cycles;
  see Remark \ref{rmk:Floer} for further discussion.
\end{Remark}

As has been pointed out above, the other result of this paper, also
established using some of the machinery utilized in the proof of
Theorem \ref{thm:main}, relates the growth of periodic orbits to the
decay of the mean indices for a certain type of periodic orbits (the
so-called carriers of the action selector) for symplectically
aspherical manifolds. In particular, we show that whenever the mean
indices of the carriers are bounded away from zero, the number of
simple periodic orbits grows linearly with the order of iteration.

The paper is organized as follows. In Section \ref{sec:prelim}, we set
our conventions and notation and discuss some standard notions and
results from symplectic topology, needed for the proof of the main
theorem. These include the mean index, the filtered and local Floer
homology, the action selector, and their properties. The main
objective of Section \ref{sec:proof} is to prove Theorem
\ref{thm:main}. The proof hinges on the notion of a carrier of the
action selector discussed in Section \ref{sec:carriers}. The relation
between the growth of orbits and the decay of mean indices is
established in Section \ref{sec:growth}.

\subsection{Acknowledgments} The authors are grateful to Alberto
Abbondandolo, Jan\-ko Latschev and Michael Usher for useful
discussions. A part of this work was carried out while both of the
authors were visiting MSRI during the Symplectic and Contact Geometry
and Topology program. The authors would like to thank MSRI for its
warm hospitality and support.

\section{Preliminaries}
\label{sec:prelim}
The goal of this section is to set notation and conventions and to
give a brief review of Floer homology and several other notions used
in the paper.

\subsection{Conventions and notation}
Let $(M^{2n},\omega)$ be a closed symplectic manifold. Throughout the
paper we will usually assume, for the sake of simplicity, that $M$ is
\emph{rational}, i.e., the group $\left<[\omega],
  {\pi_2(M)}\right>\subset\R$ formed by the integrals of $\omega$ over
the spheres in $M$ is discreet. This condition is obviously satisfied
when $M$ is \emph{negative monotone} as in Theorem \ref{thm:main} or
\emph{monotone}, i.e., $[\omega]\mid_{\pi_2(M)}=\lambda
c_1(M)\!\mid_{\pi_2(M)}$ for some $\lambda<0$ in the former case or
$\lambda\geq 0$ in the latter. Recall also that $M$ is called
\emph{symplectically aspherical} if
$[\omega]\mid_{\pi_2(M)}=0=c_1(M)\!\mid_{\pi_2(M)}$.

All Hamiltonians $H$ on $M$ considered in this paper are assumed to be
$k$-periodic in time, i.e., $H\colon S^1_k\times M\to\R$, where
$S^1_k=\R/k\Z$, and the period $k$ is always a positive integer.  When
the period is not specified, it is equal to one, which is the default
period in this paper. We set $H_t = H(t,\cdot)$ for $t\in
S^1=\R/\Z$. The Hamiltonian vector field $X_H$ of $H$ is defined by
$i_{X_H}\omega=-dH$. The (time-dependent) flow of $X_H$ will be
denoted by $\varphi_H^t$ and its time-one map by $\varphi_H$. Such
time-one maps are referred to as \emph{Hamiltonian diffeomorphisms}.
A one-periodic Hamiltonian $H$ can always be treated as
$k$-periodic. In this case, we will use the notation $H^{\# k}$ and,
abusing terminology, call $H^{\# k}$ the $k$th iteration of $H$.

Let $K$ and $H$ be one-periodic Hamiltonians such that $K_1=H_0$ and
$H_1=K_0$. We denote by $K\# H$ the two-periodic Hamiltonian equal to
$K_t$ for $t\in [0,\,1]$ and $H_{t-1}$ for $t\in [1,\,2]$. Thus, $H^{\#
  k}=H\#\ldots\# H$ ($k$ times).

Let $x\colon S^1_k\to W$ be a contractible loop. A \emph{capping} of
$x$ is a map $u\colon D^2\to M$ such that $u\mid_{S^1_k}=x$. Two
cappings $u$ and $v$ of $x$ are considered to be equivalent if the
integrals of $\omega$ and $c_1(TM)$ over the sphere obtained by
attaching $u$ to $v$ are equal to zero. A capped closed curve
$\bar{x}$ is, by definition, a closed curve $x$ equipped with an
equivalence class of cappings. In what follows, the presence of
capping is always indicated by a bar.

The action of a one-periodic Hamiltonian $H$ on a capped closed curve
$\bar{x}=(x,u)$ is defined by
$$
\CA_H(\bar{x})=-\int_u\omega+\int_{S^1} H_t(x(t))\,dt.
$$
The space of capped closed curves is a covering space of the space of
contractible loops and the critical points of $\CA_H$ on the covering
space are exactly capped one-periodic orbits of $X_H$. The \emph{action
  spectrum} $\CS(H)$ of $H$ is the set of critical values of
$\CA_H$. This is a zero measure set; see, e.g.,
\cite{HZ,schwarz}. When $M$ is rational, $\CS(H)$ is a closed, and
hence nowhere dense, set. Otherwise, $\CS(H)$ is dense in
$\R$. These definitions extend to $k$-periodic
orbits and Hamiltonians in an obvious way. Clearly, the action
functional is homogeneous with respect to iteration:
$$
\CA_{H^{\# k}}(\bx^k)=k\CA_H(\bx).
$$
Here $\bx^k$ stands for the $k$th iteration of the capped orbit $\bx$.

All results of this paper concern only contractible periodic orbits
and throughout of the paper \emph{a periodic orbit is always assumed
  to be contractible, even if this is not explicitly stated}.

A periodic orbit $x$ of $H$ is said to be \emph{non-degenerate} if the
linearized return map $d\varphi_H \colon T_{x(0)}W\to T_{x(0)}W$ has
no eigenvalues equal to one. Following \cite{SZ}, we call $x$
\emph{weakly non-degenerate} if at least one of the eigenvalues is
different from one. A Hamiltonian is non-degenerate if all its
one-periodic orbits are non-degenerate.

Let $\bar{x}$ be a non-degenerate (capped) periodic orbit.  The
\emph{Conley--Zehnder index} $\MUCZ(\bar{x})\in\Z$ is defined, up to a
sign, as in \cite{Sa,SZ}. (Sometimes, we will also use the notation
$\MUCZ(H,\bx)$.) More specifically, in this paper, the Conley--Zehnder
index is the negative of that in \cite{Sa}. In other words, we
normalize $\MUCZ$ so that $\MUCZ(\bar{x})=n$ when $x$ is a
non-degenerate maximum (with trivial capping) of an autonomous
Hamiltonian with small Hessian. The \emph{mean index}
$\Delta_H(\bx)\in\R$ measures, roughly speaking, the total angle swept
by certain eigenvalues with absolute value one of the linearized flow
$d\varphi^t_H$ along $x$ with respect to the trivialization associated
with the capping; see \cite{Lo,SZ}. The mean index is defined
regardless of whether $x$ is degenerate or not and $\Delta_H(\bx)$
depends continuously on $H$ and $\bx$ in the obvious sense. When $x$
is non-degenerate, we have
\begin{equation}
\label{eq:mean-cz}
0<|\Delta_H(\bx)-\MUCZ(H,\bx)|<n.
\end{equation}
Furthermore, the mean index is homogeneous with respect to iteration:
$$
\Delta_{H^{\# k}}(\bx^k)=k\Delta_H(\bx).
$$

\subsection{Floer homology}
In this subsection, we very briefly recall, mainly to set notation,
the construction of the filtered Floer homology. We refer the reader
to, e.g., \cite{HS,MS,Sa,SZ} and also \cite{FO,LT} for detailed
accounts and additional references.

Fix a ground field $\F$. Let $H$ be a non-degenerate Hamiltonian on
$M$.  Denote by $\CF^{(-\infty,\, b)}_k(H)$, where $b\in
(-\infty,\,\infty]$ is not in $\CS(H)$, the vector space of formal
sums
$$ 
\sigma=\sum_{\bar{x}\in \bPP(H)} \alpha_{\bar{x}}\bar{x}. 
$$
Here $\alpha_{\bar{x}}\in\F$ and $\MUCZ(\bar{x})=k$ and
$\CA_H(\bar{x})<b$. Furthermore, we require, for every $a\in \R$, the
number of terms in this sum with $\alpha_{\bar{x}}\neq 0$ and
$\CA_H(\bar{x})>a$ to be finite. The graded $\F$-vector space
$\CF^{(-\infty,\, b)}_*(H)$ is endowed with the Floer differential
counting the anti-gradient trajectories of the action functional; see,
e.g., \cite{HS,MS,Ono:AC,Sa} and also \cite{FO,LT}. Thus, we obtain a
filtration of the total Floer complex $\CF_*(H):=\CF^{(-\infty,\,
  \infty)}_*(H)$. Furthermore, we set $\CF^{(a,\,
  b)}_*(H):=\CF^{(-\infty,\, b)}_*(H)/\CF^{(-\infty,\,a)}_*(H)$, where
$-\infty\leq a<b\leq\infty$ are not in $\CS(H)$. The resulting
homology, the \emph{filtered Floer homology} of $H$, is denoted by
$\HF^{(a,\, b)}_*(H)$ and by $\HF_*(H)$ when
$(a,\,b)=(-\infty,\,\infty)$.

Working over $\F=\Z_2$, we can view a chain $\sigma=\sum
\alpha_{\bx}\bx\in \CF_*(H)$ as simply a collection of periodic orbits
$\bx$ for which $\alpha_{\bx}\neq 0$. In general, will say that $\bx$
\emph{enters} the chain $\sigma$ when $\alpha_{\bx}\neq 0$. Note also
that every $\F$-vector space $\CF_k(H)$ is finite-dimensional when $M$
is negative monotone or monotone with $\lambda>0$.

The total Floer complex and homology are modules over the
\emph{Novikov ring} $\Lambda$. In this paper, the latter is defined as
follows. Let $\omega(A)$ and $\left<c_1(TM),A\right>$ denote the
integrals of $\omega$ and, respectively, $c_1(TM)$ over a cycle $A$.
Set
$$
I_\omega(A)=-\omega(A)\text{ and } I_{c_1}(A)=-2\left<c_1(TM),
  A\right>,
$$
where $A\in\pi_2(M)$.  For instance,
$$
I_\omega=\frac{\lambda}{2}I_{c_1}
$$
when $M$ is monotone or negative monotone. In particular,
\emph{$I_\omega(A)$ and $I_{c_1}(A)$ have opposite signs when $M$ is
  negative monotone}.  Let
$$
\Gamma=\frac{\pi_2(M)}{\ker I_\omega\cap \ker I_{c_1}}.
$$
Thus, $\Gamma$ is the quotient of $\pi_2(M)$ by the equivalence
relation where the two spheres $A$ and $A'$ are considered to be
equivalent if $\omega(A)=\omega(A')$ and $\left<c_1(TM),
  A\right>=\left<c_1(TM), A'\right>$. For example, $\Gamma\simeq \Z$
when $M$ is negative monotone or monotone with $\lambda\neq 0$.  The
homomorphisms $I_\omega$ and $I_{c_1}$ descend to $\Gamma$ from
$\pi_2(M)$.

The group $\Gamma$ acts on $\CF_*(H)$ and on $\HF_*(H)$ via recapping:
an element $A\in \Gamma$ acts on a capped one-periodic orbit $\bar{x}$
of $H$ by attaching the sphere $A$ to the original capping. We denote
the resulting capped orbit by $\bx\# A$. Then,
$$
\MUCZ(\bx\# A)=\MUCZ(\bx)+ I_{c_1}(A)
\text{ and }
\CA_H(\bx\# A)=\CA_H(\bx)+I_\omega(A).
$$
Note that in a similar vein we also have
\begin{equation}
\label{eq:delta}
\Delta_H(\bx\# A)=\Delta_H(\bx)+ I_{c_1}(A),
\end{equation}
regardless of whether $x$ is non-degenerate or not.

The Novikov ring $\Lambda$ is a certain completion of the group ring
of $\Gamma$ over $\F$. Namely, $\Lambda$ is comprised of formal linear
combinations $\sum \alpha_A e^A$, where $\alpha_A\in\F$ and $A\in
\Gamma$, such that for every $a\in \R$ the sum contains only finitely
many terms with $I_\omega(A) > a$ and, of course, $\alpha_A\neq
0$. The Novikov ring $\Lambda$ is graded by setting
$\deg(e^A)=I_{c_1}(A)$ for $A\in\Gamma$.  The action of $\Gamma$ turns
$\CF_*(H)$ and $\HF_*(H)$ into $\Lambda$-modules.

The definition of Floer homology extends to all, not necessarily
non-degenerate, Hamiltonians by continuity.  Let $H$ be an arbitrary
(one-periodic in time) Hamiltonian on $M$ and let the end points $a$
and $b$ of the action interval be outside $\CS(H)$. By definition, we
set
$$
\HF^{(a,\, b)}_*(H)=\HF^{(a,\, b)}_*(\tH),
$$
where $\tH$ is a non-degenerate, small perturbation of $H$. It is well
known that here the right hand side is independent of $\tH$ as long as
the latter is sufficiently close to $H$.  Working with filtered Floer
homology, \emph{we will always assume that the end points of the
  action interval are not in the action spectrum.} (At this point the
background assumption that $M$ is rational becomes essential; see \cite{He:irr}
for the irrational case and also \cite[Remark 2.3]{GG:gaps}.)

The total Floer homology is independent of the Hamiltonian and, up to
a shift of the grading and the effect of recapping, is isomorphic to
the homology of $M$. More precisely, we have
$$
\HF_*(H)\cong \H_ {*+n}(M;\F)\otimes \Lambda
$$
as graded $\Lambda$-modules.

\begin{Remark}
\label{rmk:Floer}
We conclude this discussion by recalling that in order for the Floer
differential to be defined certain regularity conditions must be
satisfied generically. To ensure this, we have to either require $M$ to
be weakly monotone (see \cite{HS,MS,Ono:AC,Sa}) or utilize the
machinery of virtual cycles (see \cite{FO,FOOO,LT} or, for the
polyfold approach, \cite{HWZ,HWZ2} and references therein). In the
latter case the ground field $\F$ is required to have zero
characteristic.  Here we are primarily interested in negative
monotone manifolds.  Such a manifold is weakly monotone if and only if
$N\geq n-2$, where $N$ is the minimal Chern number, i.e., the positive
generator of $\left<c_1(TM),\pi_2(M)\right>\subset\Z$.
\end{Remark}

\subsection{Local Floer homology}
\label{sec:LFH}
In this section, we briefly recall the definition and basic properties
of local Floer homology, following mainly \cite{GG:gaps} although this
notion goes back to the original work of Floer (see, e.g.,
\cite{F:witten,Fl}) and has been revisited a number of times since
then; see, e.g., \cite[Section 3.3.4]{Poz} and
\cite{Gi:conley,GG:gap,He:irr}.

Let $\bx$ be a capped isolated one-periodic orbit of a Hamiltonian
$H\colon S^1\times M\to \R$. Pick a sufficiently small tubular
neighborhood $U$ of $x$ and consider a non-degenerate $C^2$-small
perturbation $\tH$ of $H$.  (Strictly speaking $U$ is a neighborhood
of the graph of the orbit in the extended phase space $S^1\times M$.)
The orbit $x$ splits into non-degenerate
one-periodic orbits of $\tH$, which are $C^1$-close to $x$ and hence
contained in $U$. The capping of $\bx$ gives rise to the cappings of
these orbits in an obvious way and the action of $\tH$ on the
resulting capped orbits is close to $\CA_H(\bx)$.  

Every (anti-gradient) Floer trajectory $u$ connecting such capped
one-periodic orbits of $\tH$ lying in $U$ is also contained in $U$,
provided that $\|\tH-H\|_{C^2}$ is small enough.  Thus, by the
compactness and gluing theorems, every broken anti-gradient trajectory
connecting two such orbits also lies entirely in $U$. Similarly to the
definition of the ordinary Floer homology, consider the complex
$\CF_*(\tH,\bx)$ over $\F$ generated by the capped one-periodic orbits
of $\tH$ in $U$ with cappings inherited from $\bx$, graded by the
Conley--Zehnder index and equipped with the Floer
differential defined in the standard way.  The
continuation argument (see, e.g., \cite{MS,SZ}) shows that the
homology of this complex is independent of the choice of $\tH$ and of
other auxiliary data (e.g., an almost complex structure). We refer to
the resulting homology group $\HF_*(H,\bx)$ as the \emph{local Floer
  homology} of $H$ at $\bx$. This definition is a Floer theoretic
analogue of the local Morse homology or critical modules; see \cite{GrMe}
and also, e.g., \cite[Section 14.1]{Lo} and references therein.

\begin{Example}
  Assume that $\bx$ is non-degenerate and $\MUCZ(\bx)=k$.  Then
  $\HF_l(H,\bx)=\F$ when $l=k$ and $\HF_l(H,\bx)=0$ otherwise.
\end{Example}

\begin{Remark}
  In \cite{Gi:conley,GG:gaps,GG:gap}, the perturbation $\tH$ is
  required to be supported in $U$ or, more precisely,
  $\supp(H-\tH)\subset U$. This requirement is immaterial and we omit it
  in this paper. Furthermore we would like to emphasize that the above
  construction is genuinely local. At no point it depends on global
  properties of $M$ (such as weak monotonicity) or requires virtual
  cycles. In fact, it suffices to have the Hamiltonian $H$ defined
  only on a neighborhood of the orbit. However, a capping of the orbit
  is still essential or, at least, a trivialization of $TM$ along the
  orbit must be fixed.
\end{Remark}

Let us now briefly review the basic properties of local Floer homology
that are essential for what follows, assuming for the sake of
simplicity that $M$ is rational.

Just as the ordinary Floer homology, the local Floer homology is
homotopy invariant. To be more precise, let $H_s$, $s\in [0,\, 1]$, be
a family of Hamiltonians such that $x$ is a uniformly isolated
one-periodic orbit for all $H^s$, i.e., $x$ is the only periodic orbit
of $H_s$, for all $s$, in some open set independent of $s$. Then
$\HF_*(H^s,\bx)$ is constant throughout the family:
$\HF_*(H^0,\bx)=\HF_*(H^1,\bx)$.

Local Floer homology spaces are building blocks for filtered Floer
homology. Namely, assume that all one-periodic orbits $x$ of $H$ are
isolated and $\HF_k(H,\bx)=0$ for some $k$ and all $\bx$. Then
$\HF_k(H)=0$. Moreover, let $c\in \R$ be such that all capped
one-periodic orbits $\bx_i$ of $H$ with action $c$ are isolated. (As a
consequence, there are only finitely many orbits with action close to
$c$.) Then, if $\eps>0$ is small enough,
$$
\HF_*^{(c-\eps,\,c+\eps)}(H)=\bigoplus_i \HF_*(H,\bx_i).
$$

By definition, the \emph{support} of $\HF_*(H,\bx)$ is the collection
of integers $k$ such that $\HF_k(H,\bx)\neq 0$. By \eqref{eq:mean-cz}
and continuity of the mean index, $\HF_*(H,\bx)$ is supported in the
range $[\Delta_H(\bx)-n,\, \Delta_H(\bx)+n]$. Moreover, when $x$ is
weakly non-degenerate, the support is contained in the open interval
$(\Delta_H(\bx)-n,\, \Delta_H(\bx)+n)$; cf.\ \cite{SZ}.

\subsection{Action selectors}
The theory of Hamiltonian \emph{action selectors} or \emph{spectral
  invariants}, as they are usually referred to, was developed in its
present Floer--theoretic form in \cite{Oh,schwarz} although the first
versions of the theory go back to \cite{HZ,Vi:gen}. (See also, e.g.,
\cite{EP,FS,FGS,Gi:we,Gu,MS,U1,U2} and references therein for other
versions, some of the applications and a detailed discussion of the
theory; here, we mainly follow \cite{GG:gaps}.)

Let $M$ be a closed symplectic manifold and let $H$ be a Hamiltonian
on $M$. We assume that $M$ is rational --
this assumption greatly simplifies the theory (cf.\ \cite{U1}) and is
obviously satisfied for negative monotone manifolds.

The \emph{action selector} $\s$ associated with the fundamental class
$[M]\in \H_{2n}(M;\F) \subset \HF_n(H)$ is defined as
$$
\s(H)= \inf\{ a\in \R\ssminus \CS(H)\mid u\in \im(i^a)\}
=\inf\{ a\in \R\ssminus \CS(H)\mid j^a(u)=0\},
$$
where $i^a\colon \HF_*^{(-\infty,\,a)}(H)\to \HF_*(H)$ and
$j^a\colon \HF_*(H)\to\HF_*^{(a,\, \infty)}(H)$ are the natural
``inclusion'' and ``quotient'' maps.
Then $\s(H)>-\infty$ as is easy to see; \cite{Oh}. 

The action selector $\s$ has the following properties:

\begin{itemize}

\item[(AS1)] Normalization: $\s(H)=\max H$ if $H$ is autonomous and
  $C^2$-small.
             
\item[(AS2)] Continuity: $\s$ is Lipschitz in $H$ in the
  $C^0$-topology.

\item[(AS3)] Monotonicity: $\s(H)\geq \s(K)$ whenever $H\geq K$
  pointwise.

\item[(AS4)] Hamiltonian shift: $\s(H+a(t))=\s(H)+\int_0^1a(t)\,dt$,
  where $a\colon S^1\to\R$.

\item[(AS5)] Symplectic invariance: $\s(H)=\s(\varphi^* H)$ for any
  symplectomorphism $\varphi$.

\item[(AS6)] Homotopy invariance: $\s(H)=\s(K)$ when
  $\varphi_H=\varphi_K$ in the universal covering of the group of
  Hamiltonian diffeomorphisms and both $H$ and $K$ are normalized to
  have zero mean.

\item[(AS7)] Triangle inequality or sub-additivity: $\s(H\#
  K)\leq\s(H)+\s(K)$.

\item[(AS8)] Spectrality: $\s(H)\in \CS(H)$.  More specifically, there
  exists a capped one-periodic orbit $\bx$ of $H$ such that
  $\s(H)=\CA_H(\bx)$.

\end{itemize}

This list of the properties of $\s$ is far from exhaustive, but it is
more than sufficient for our purposes. Most of the properties are
implicitly used in the proof of Lemma \ref{lemma:key} below and the
sub-additivity in the form
\begin{equation}
\label{eq:sub-add}
\s(H^{\# k})\leq k\s(H)
\end{equation} 
is crucial for the proof of Theorem \ref{thm:main}.  It is
worth emphasizing the rationality assumption plays an important role
in the proofs of the homotopy invariance and spectrality; see
\cite{Oh,schwarz} and also \cite{EP} for a simple proof. (The latter
property also holds in general for non-degenerate Hamiltonians. This
is a non-trivial result; \cite{U1}.)  Finally note that for the
triangle inequality to hold one has to work with a suitable definition
of the pair-of-pants product in Floer homology; cf.\ \cite{AS,U2}. We
refer the reader to \cite{U2} for a very detailed treatment of action
selectors in generality more than sufficient for our purposes.

\section{Proof of Theorem \ref{thm:main}}
\label{sec:proof}
As most of the proofs of the Conley conjecture type results, the proof
of Theorem \ref{thm:main} amounts to dealing with two
cases. The so-called degenerate case is established for all rational
symplectic manifolds in \cite{GG:gaps}. Here we consider the second,
the non-degenerate case.  (This terminology should be taken with a
grain of salt. For, for instance, in the non-degenerate case only some
orbits of the Hamiltonian must be non-degenerate, and in fact only
weakly non-degenerate. Note also that, when $c_1(TM)|_{\pi_2(M)}=0$,
this case was settled already in \cite{SZ}.) 

To illustrate the idea of
the proof, let us consider the following situation. Assume that an
orbit $x$ and its sufficiently large iteration $x^k$, equipped
with suitable cappings, ``represent'' the fundamental class in the
Floer homology. Denote
the corresponding capped orbits by $\bx$ and $\by=\bx^k\# A$. To
represent the fundamental class, the orbits $\bx$ and $\by$ must at
least have non-vanishing Floer homology in degree $n$. When $x$ is
weakly non-degenerate, we have $\Delta_H(\bx)>0$. Hence,
$\HF_n(H,\by)\neq 0$ implies, by \eqref{eq:delta}, that $I_{c_1}(A)<0$
for large $k$. On the other hand, due to the sub-additivity of the
action selector, we have
$$
k\CA_H(\bx)+I_\omega(A)=\CA_{H^{\# k}}(\by)=\s(H^{\# k})\leq k\s(H)=k\CA_H(\bx).
$$
Hence, $I_\omega(A)\leq 0$, which is impossible since for negative
monotone manifolds $I_{c_1}$ and $I_\omega$ have opposite signs. The
remaining case where $x$ is strongly degenerate and $\Delta_H(\bx)=0$
is considered in \cite{GG:gaps}.

To make this argument rigorous, we need to make sense 
of the statement that the orbits represent the fundamental class. This
is done using the notion of a carrier of the action selector,
discussed in the next section.

\subsection{Carrier of the action selector} 
\label{sec:carriers}
When $H$ is non-degenerate, the action selector can also be evaluated
as
$$
\s(H)=\inf_{[\sigma]=[M]}\CA_H(\sigma),
$$
where we set 
$$
\CA_H(\sigma)=\max\{\CA_H(\bx)\mid \alpha_{\bx} \neq 0\}\text{ for }
\sigma=\sum\alpha_{\bx} \bx\in\CF_n(H).
$$
The infimum here is obviously, when $M$ is rational, attained.  Hence
there exists a cycle $\sigma=\sum\alpha_{\bx} \bx\in\CF_n(H)$,
representing the fundamental class $[M]$, such that $\s(H)=\CA_H(\bx)$
for an orbit $\bx$ entering $\sigma$. In other words, $\bx$ maximizes
the action on $\sigma$ and the cycle $\sigma$ minimizes the action
over all cycles in the homology class $[M]$. We call such an orbit
$\bx$ a \emph{carrier} of
the action selector. Note that this is a stronger requirement than
just the equality $\s(H)=\CA_H(\bx)$. A carrier is not in general
unique, but it becomes unique when all one-periodic orbits of $H$ have
distinct action values.

Our goal is to generalize this definition to the case where all
one-periodic orbits of $H$ are isolated but possibly degenerate. Under
a $C^2$-small, non-degenerate perturbation $\tH$ of $H$, every such
orbit $x$ splits into several non-degenerate orbits, which are
close to $x$. Furthermore, a capping of $x$ naturally gives rise to a
capping of each of these orbits; cf.\ Section \ref{sec:LFH}.

\begin{Definition}
\label{def:carrier}
A capped one-periodic orbit $\bx$ of $H$ is a carrier of the action
selector for $H$ if there exists a sequence of $C^2$-small,
non-degenerate perturbations $\tH_i\stackrel{C^2}{\to} H$ such that
one of the capped orbits $\bx$ splits into is a carrier for
$\tH_i$. An orbit (without capping) is said to be a carrier if it
turns into one for a suitable choice of capping.
\end{Definition}

It is easy to see that a carrier necessarily exists, provided that $M$
is rational and all one-periodic orbits of $H$ are isolated.  As in
the non-degenerate case, a carrier is of course not unique in general
-- different choices of sequences $\tH_i$ can lead to different
carriers. However, it becomes unique when all one-periodic orbits of
$H$ have distinct action values. In other words, under the latter
requirement, the carrier is independent of the choice of the sequence
$\tH_i$.

Note also that one can always arrange for the sequence $\tH_i$ to
satisfy the additional condition
\begin{equation}
\label{eq:carrier-action}
\s(H)=\s(\tH_i)
\end{equation}
by adding a small constant to $\tH_i$. We can further assume that all
one-periodic orbits of $\tH_i$ have distinct action values. (This is a
consequence of the fact that the Floer complex is stable under small
perturbations of the Hamiltonian and auxiliary structures when all
regularity requirements are met.) In what follows, we will always pick
$\tH_i$ such that \eqref{eq:carrier-action} holds and the latter
condition is satisfied.

As an immediate consequence of the definition of the carrier and
continuity of the action and the mean index, we have
\begin{equation}
\label{eq:range}
\s(H)=\CA_H(\bx)\text{ and } 0\leq \Delta_H(\bx) \leq 2n,
\end{equation}
and the inequalities are strict when $x$ is weakly non-degenerate. 

We will need the following result asserting that a carrier of the
action selector is in some sense homologically essential.

\begin{Lemma} 
\label{lemma:key}
Assume that all one-periodic orbits of $H$ are isolated and let $\bx$
be a carrier of the action selector. Then $\HF_n(H,\bx)\neq 0$.
\end{Lemma}

\begin{proof}
  Let $\tH$ be a one of the Hamiltonians in the sequence $\tH_i$ and
  let $\by$ be a carrier of the action selector for $\tH$ in the
  collection $Y$ of non-degenerate orbits which $\bx$ splits
  into. Thus, $\by\in Y$ is an action maximizing orbit in an action
  minimizing cycle $\sigma=\by+\ldots\in \CF_n(\tH)$. As has been
  pointed out above, we may assume that $\by$ is a unique action
  maximizer in $\sigma$ and that $\CA_{\tH}(\by)=\s(H)$.  

  At this point it is convenient to specify how close $\tH$ must be to
  $H$.  Fix a small, isolating neighborhood $U$ of $x$ and let $V$ be
  the union of small neighborhoods of the remaining one-periodic
  orbits of $H$. (In particular, $U$ and $V$ are disjoint.) Thus, for
  instance, the group $Y$ comprises one periodic orbits of $\tH$
  contained in $U$ and equipped with the cappings inherited from
  $\bx$. (Strictly speaking, here, as in the definition of the local
  Floer homology, one should work with tubular neighborhoods of
  periodic orbits in the extended phase-space $S^1\times M$.) It is
  easy to show that for every $\tH$, which is $C^2$-close to $H$, all
  one-periodic orbits of $\tH$ are contained in $U\cup V$ and,
  moreover, every solution $u$ of the Floer equation connecting an
  orbit in $U$ and an orbit in $V$ must have energy $E(u)> \eps$ for
  some constant $\eps>0$ independent of $u$ and $\tH$; see, e.g.,
  \cite[Section 1.5]{Sa} or \cite[Lemma 19.8]{FO}. This is our first
  closeness requirement. Furthermore, we pick $\tH$ so $C^2$-close to
  $H$ that every orbit from $Y$ has action in the range
  $(\s(H)-\eps/2,\, \s(H)+\eps/2)$.

  Denote by $\sigma_Y$ the chain formed by the orbits from $Y$
  entering the cycle $\sigma$. For instance, $\by$ enters the chain
  $\sigma_Y$. Our goal is to prove that $\sigma_Y\in \CF_*(\tH,\bx)$
  is closed but not exact.

  First let us show that $\sigma_Y$ is in fact a cycle in
  $\CF_*(\tH,\bx)$.  To see this, denote by $\p_Y \sigma_Y$ the part
  of the cycle $\p \sigma_Y$ formed by the orbits that also belong to
  $Y$. We need to show that $\p_Y\sigma_Y=0$. Assume the contrary: an
  orbit $\bz'$ enters this cycle. Then, since $\sigma$ is closed,
  there must be an orbit $\bz\not\in Y$ in $\sigma$ connected to
  $\bz'$ by a Floer trajectory $u$. In particular, $u$ connects an
  orbit in $V$ to an orbit in $U$ and hence $E(u)>\eps$. Thus
$$
\CA_{\tH}(\bz)\geq \CA_{\tH}(\bz')+E(u)>
\s(H)-\frac{\eps}{2}+\eps>\s(H)=\CA_{\tH}(\by).
$$
This is impossible since $\by$ is an action maximizing orbit in the
cycle $\sigma$.
 
To finish the proof, it remains to show that $\sigma_Y$ is not exact
in $\CF_*(\tH,\bx)$. Assume the contrary. Then there exists a chain
$\eta$ formed by some orbits from $Y$ such that $\p_Y\eta=\sigma_Y$ in
obvious notation. We have $\p \eta=\p_Y \eta+\lambda$, where all
orbits entering $\lambda$ are contained in $V$, while those from $\p_Y
\eta$ are contained in $U$. By the definition of the Floer
differential, any orbit $\bz$ entering $\lambda$ is connected to an
orbit entering $\eta$ by a Floer trajectory $u$ beginning in $U$ and
ending in $V$. Hence,
$$
\CA_{\tH}(\bz)\leq \ \s(H)+\eps/2-\eps<\s(H)=\CA_{\tH}(\sigma).
$$
As a consequence, every orbit in the cycle 
$$
\sigma'=\sigma-\p \eta = (\sigma-\sigma_Y)-\lambda
$$ 
has action strictly smaller than $\CA_{\tH}(\sigma)$.  (Indeed, an
orbit from $\sigma'$ either enters $\sigma$, but not $Y$, or enters
$\lambda$. In the latter case, the action does not exceed
$\CA_{\tH}(\sigma)-\eps/2$ as we have just shown. In the former case,
the action is again smaller than
$\CA_{\tH}(\sigma)=\CA_{\tH}(\by)$. For $\by$ is a unique action
maximizing orbit in $\sigma$.)  To summarize, we have $[\sigma']=[M]$
and $\CA_{\tH}(\sigma')<\CA_{\tH}(\sigma)$. This contradicts our
choice of $\sigma$ as an action minimizing cycle.
\end{proof}

\begin{Remark}
  We finish this discussion with one minor, fairly standard, technical
  point. Namely, recall that the Floer complex of a non-degenerate
  Hamiltonian $H$ depends not only on $H$ but also on an auxiliary
  structure $J$, e.g., an almost complex structure when $M$ is weakly
  monotone. Moreover, the complex is defined only when suitable
  regularity requirements are met. As a consequence, an action carrier
  is in reality assigned to the pair $(H,J)$ rather than to just a
  Hamiltonian $H$ in either non-degenerate or degenerate case. Thus,
  in Definition \ref{def:carrier}, we tacitly assumed the presence of
  an auxiliary structure $J$ in the background and that the regularity
  requirements are satisfied for the sequence of perturbations. This
  can be achieved by either considering regular pairs $(\tH_i,J_i)$
  with $J_i\to J$ or even by setting $J_i=J$; cf.\ \cite{FHS,SZ}.
\end{Remark}

\subsection{Proof}
Arguing by contradiction, assume that $H$ has finitely many simple,
i.e., non-iterated, periodic orbits. By passing if necessary to an
iteration (and ``adjusting the time''), we can also assume that all
periodic orbits are one-periodic.

\subsubsection{A common action carrier} 
\label{sec:step1}
Our next goal is to show that after passing if necessary to an
iteration we can assume that there exists a one-periodic orbit $x$
(not capped) of $H$ such that its iterations $x^{k_i}$ are action
carriers for $H^{\# k_i}$ for some infinite sequence $k_1=1, k_2,
k_3,\ldots$.

Indeed, let $x_1,\ldots,x_m$ be all periodic (and thus in fact
one-periodic) orbits of $H$. Let us break down the set of positive
integers $\N$ into $m$ groups $Z_1,\ldots,Z_m$ by assigning
$k\in\N$ to $Z_j$ if $x_j$ is an action selector carrier for
$H^{\# k}$. If the action carrier for $H^{\# k}$ is not unique, we pick
it in an arbitrary fashion. Thus, we have
$$
\N=\bigsqcup_{j=1}^m Z_j
$$ 
with some of the sets $Z_j$ possibly empty. By the pigeonhole
principle, one of the sets $Z_j$, say $Z_1$, contains an infinite
subsequence $2^{l_s}$ of the sequence $k=2^l$ of iterations. Let us
replace $H$ by $H^{\# 2^{l_1}}$ and set $x=x_1^{2^{l_1}}$, adjusting
the time again. Then $x^{k_i}$ is a carrier for $H^{\# k_i}$, where
$k_1=1, k_2=2^{l_2-l_1}, k_3= 2^{l_3-l_1},\ldots$.

\subsubsection{Two alternatives and the degenerate case}
Let $\bx$ be the orbit $x$ capped as a carrier of $H$. Then
$\Delta_H(\bx)\geq 0$ and $\HF_n(H,\bx)\neq 0$ by \eqref{eq:range} and
Lemma \ref{lemma:key}. When $\Delta_H(\bx)=0$, the capped orbit $\bx$
is the so-called \emph{symplectically degenerate maximum} and in the
presence of such an orbit the Conley conjecture is proved for rational
symplectic manifolds in \cite{GG:gaps}. (See also
\cite{Gi:conley,GG:gap,He:irr} for the definition, a detailed
discussion and applications of this notion, which originates from
Hingston's proof of the Conley conjecture for tori; see \cite{Hi}.)

Thus, it remains to prove the theorem in the case where
$\Delta_H(\bx)>0$ which is done in the next subsection. Here we only
mention that this is the so-called \emph{weakly non-degenerate
  case}. For, as is shown in \cite{GG:gap}, whenever $\Delta_H(\bx)>
0$ and $\HF_n(H,\bx)\neq 0$, the orbit $x$ is weakly non-degenerate.

\subsubsection{Weakly non-degenerate case} Pick a sufficiently large
entree $k$ in the sequence $k_i$ from Section \ref{sec:step1} so that
$k\Delta_H(\bx)> 2n$. Then $n$ is not in the support of
$\HF_*(H,\bx^k)$. The capped orbit $\bx^k$ cannot be a carrier of the
action selector
for $H^{\# k}$ by \eqref{eq:range} and thus it becomes one after
non-trivial recapping by some sphere $A$. Denote by $\by=\bx^k\# A$
the resulting carrier for $H^{\# k}$. Applying \eqref{eq:delta} and
\eqref{eq:range} to $\by$, we have
$$
2n+I_{c_1}(A) < k\Delta_H(\bx)+I_{c_1}(A)=\Delta_{H^{\# k}}(\by)\leq 2n.
$$
Hence,
$$
I_{c_1}(A)<0.
$$
On the other hand, using the sub-additivity \eqref{eq:sub-add} of the
action selector, we obtain
$$
k\CA_{H}(\bx)+I_\omega(A)=\CA_{H^{\# k}}(\by)=\s(H^{{\# k}})\leq
k\s(H)=k\CA_H(\bx).
$$
Therefore, 
$$
I_\omega(A)\leq 0.
$$
However, we cannot have $I_{c_1}(A)<0$ and $I_\omega(A)\leq 0$
simultaneously for a negative monotone manifold $M$, since for $M$, as
has been mentioned above, $I_\omega=\lambda I_{c_1}$ with
$\lambda<0$. This contradiction completes the proof.

\subsection{Growth of orbits vs.\ decay of the mean indexes}
\label{sec:growth}
The notion of an action carrier also lends itself readily to the proof
of a simple result relating, in the symplectically aspherical case,
the growth of the number of orbits of a certain type to the decay of
the minimal mean index. To state this result, we need to change
slightly our notation and recall some standard terminology.  Let $x$
be a simple $\tau$-periodic orbit of the Hamiltonian diffeomorphism
$\varphi_H$. Thus $x$ comprises exactly $\tau$ fixed points of
$\varphi^\tau_H$ and as many $\tau$-periodic orbits of $H$. In what
follows, we do not distinguish these orbits of $H$. In other words, we
differentiate only \emph{geometrically distinct} simple periodic
orbits of $H$ or $\varphi_H$. Furthermore, we restrict our attention
only to $x$ such that the corresponding orbits of $H$ are
contractible.  Set $\Delta(x):=\Delta_{H^{\# \tau}}(x)$ and
$\HF_*(H,x):=\HF_*(H^{\# \tau},x)$, where on the right hand side we
can take any periodic orbit of $H$ representing $x$.

\begin{Proposition}
\label{prop:growth}
Let $H$ be a Hamiltonian on a closed, symplectically aspherical
manifold $M^{2n}$. Assume that all periodic orbits of $H$ are
isolated. Then
$$
2n \sum_x \frac{1}{\Delta(x)}\geq k ,
$$
where the sum is taken over all (geometrically distinct) simple
periodic orbits $x$ with period less than or equal to $k$ and
$\HF_n(H,x)\neq 0$ (and hence $0\leq \Delta(x)\leq 2n$).
\end{Proposition}

Here we use the convention that when $\Delta(x)=0$ the left hand side
is infinite and the inequality is automatically satisfied. Thus the
proposition essentially concerns only the weakly non-degenerate
case. For, when there is a symplectically degenerate maximum
($\Delta(x)=0$ and $\HF_n(H,x)\neq 0$), the result holds trivially and
gives no information.

\begin{Example}
  Assume that $\Delta(x)>\delta>0$ for all simple periodic orbits with
  $\HF_n(H,x)\neq 0$. Then the number of such (geometrically distinct)
  orbits of period up to $k$, which we denote by $P(k)$, grows at
  least linearly with $k$. More precisely, $P(k)\geq \delta k/2n$.
\end{Example}

\begin{proof}[Proof of Propostion \ref{prop:growth}]
  Fix a positive integer $k$. For a simple periodic orbit $x$ with
  period $\tau \leq k$, let $Z(x)$ be the collection of positive
  integers $\ell\leq k$, divisible by $\tau$, such that the iteration
  $x^{\ell/\tau}$ is a carrier of the action selector for $H^{\#
    \ell/\tau}$. By definition,
$$ 
\bigcup_{x} Z(x)=\{1,2,\ldots,k\}.
$$
Note that $|Z(x)|\leq 2n/\Delta(x)$. Indeed,
$\Delta_{H^{\# r\tau}}(x^r)> 2n$ when $r > 2n/\Delta(x)$, and hence
  $x^r$ cannot be a carrier of the action selector. Thus
$$
\sum_x \frac{2n}{\Delta(x)}\geq \sum_x |Z(x)|\geq k
$$ 
and the proposition follows.
\end{proof}

\begin{Remark}
  In the context of symplectic topology, little is known about the
  growth of orbits for sufficiently general Hamiltonian
  diffeomorphisms $\varphi$ of symplectically aspherical
  manifolds. (Numerous growth results obtained by dynamical systems
  methods usually rely on restrictive additional requirements on
  $\varphi$ such as, for instance, hyperbolicity or topological
  requirements; see, e.g., \cite{KH} and references therein.) In the
  setting considered here, the strongest growth result is probably the
  one from \cite{SZ} asserting the existence of a simple periodic
  orbit for every sufficiently large prime period when $H$ is weakly
  non-degenerate. Thus, in this case, $P(k)$ grows at least as fast as
  $k/\ln k$.  The only linear growth result the authors are aware of
  is that of Viterbo, \cite{Vi:gen}, according to which the number of
  simple periodic orbits with positive action and period less than or
  equal to $k$ grows at least linearly with $k$ when $H\geq 0$ is a
  compactly supported, non-zero Hamiltonian on $\R^{2n}$ (or on any
  so-called wide manifold -- see \cite{Gu}).
\end{Remark}

\end{document}